\PassOptionsToPackage{nameinlink}{cleveref}
\documentclass[a4paper,UKenglish,cleveref, hyperref,autoref, thm-restate,lstlisting]{lipics-v2021}

\usepackage{caption}

\usepackage[table, svgnames, dvipsnames]{xcolor}

\usepackage{framed}
\usepackage{pgfplots}
\pgfplotsset{compat=1.15}
\usetikzlibrary{calc,positioning}

\usepackage{todonotes}

\usepackage{url}

\hypersetup{
    colorlinks = true
}

\title{Tight Chernoff-Like Bounds under Limited Independence} 

\titlerunning{Tight Chernoff-Like Bounds under Limited Independence}

\author{Maciej Skorski}{University of Luxembourg}{maciej.skorski@gmail.com}{}{}
\funding{The research supported by the FNR grant C17/IS/11613923.}
\authorrunning{M. Skorski}

\Copyright{Maciej Skorski}

\ccsdesc[500]{Mathematics of computing~Statistical paradigms}

\keywords{concentration inequalities, tail bounds, limited independence}

\category{}

\relatedversion{}

\supplement{}

\acknowledgements{The author thanks the reviewers of RANDOM'22 for insightful comments.}

\nolinenumbers 

\EventEditors{Amit Chakrabarti and Chaitanya Swamy}
\EventNoEds{2}
\EventLongTitle{Approximation, Randomization, and Combinatorial Optimization. Algorithms and Techniques (APPROX/RANDOM 2022)}
\EventShortTitle{\mbox{\scriptsize APPROX/RANDOM 2022}}
\EventAcronym{APPROX/RANDOM}
\EventYear{2022}
\EventDate{September 19--21, 2022}
\EventLocation{University of Illinois, Urbana-Champaign, USA (Virtual Conference)}
\EventLogo{}
\SeriesVolume{245}
\ArticleNo{15}

\begin{document}

\maketitle

\begin{abstract}
This paper develops sharp bounds on moments of sums of $k$-wise independent bounded random variables, under constrained average variance. The result closes the problem addressed in part in the previous works of Schmidt et al. and Bellare, Rompel. We also discuss other applications of independent interests, such as asymptotically sharp bounds on binomial moments.
\end{abstract}

\section{Introduction}

\subsection{Motivation}
Consider sums of random variables, possibly differently distributed. What can be said about the probability distribution, particularly the tails,
if we only assume \emph{$k$-wise independence}, that is any $k$ of the $n$ summands are independent? 
Such a dependency condition is an appealing concept studied in a broad class of problems related to pseudoradomness, including hashing, random graphs, random projections and circuits~\cite{luby2006pairwise,alon2008k,kane2010derandomized,bazzi2009polylogarithmic};
specifically, concentration results under $k$-wise independence find important applications including (but not limited to) constructions of pseudorandom generators~\cite{impagliazzo2012pseudorandomness}, 
load balancing~\cite{karger2004simple,schmidt1995chernoff}, derandomization~\cite{sack1999handbook,ghaffari2018derandomizing,censor2020derandomizing}, streaming algorithms~\cite{mcgregor2019better} and cryptography~\cite{bellare1994randomness,barak2003true,dodis2014key,ball2018non}.

At first glance, the problem seems well addressed by concentration inequalities, such as the classical bounds due to Bernstein, Chernoff, Hoeffding, Bennet~\cite{bernsteinprobability,chernoff1952measure,hoeffding1963probability,bennett1962probability} or their modern sub-gaussian or sub-gamma generalizations~\cite{boucheron2013concentration} (obtained from moment generating functions); at the very least one may hope to utilize more exotic moment bounds such as Rosenthal-type inequalities~\cite{rosenthal1970subspaces,boucheron2005moment} or more general frameworks~\cite{latala1997estimation}. 
However, the exponential moment methods are inadequate for limited dependence, whereas moment methods are hard to apply for sums of heterogenic terms. The state-of-the-art is held by  the two influential works~\cite{schmidt1995chernoff,bellare1994randomness} which resort to direct moment calculations, offering bounds for certain parameter regimes.

The goal of the current paper is to establish \emph{sharp moment bounds} for sums of bounded $k$-wise independent variables, strengthening the state-of-art results~\cite{schmidt1995chernoff,bellare1994randomness}.
As in prior work, we assume that the summands are bounded and the sum variance is known. Formally:
\begin{framed}
\begin{quote}
Let $S = \sum_{i=1}^n X_i$ be a sum of $k$-wise independent random variables, possibly differently distributed. Suppose that a) $X_i \in [-1,1]$ and b) the average variance is $\frac{1}{n}\sum_{i=1}^{n}\mathbb{V}[X_i] = \sigma^2$.
What is the \emph{best} bound on moments of $S-\mathbb{E}S$?
\end{quote}
\end{framed}
Answering this question obviously gives desired Chernoff-like tail bounds, via an application of Markov's inequality. This approach, called the moment method, is the state-of-art technique of establishing tail bounds~\cite{boucheron2013concentration,latala1997estimation}, even superior to the exponential moment method~\cite{philips1995moment}, so it should give us as much as we can get. 
\subsection{Our Contribution}

The novelty of this work has the following aspects
\begin{itemize}
\item sharp bounds are found for \emph{all} parameter regimes, which improves upon prior works
\item some elegant techniques, novel in this context, are demonstrated; particularly the powerful method of \emph{symmetrization} from high-dimensional probability~\cite{vershynin2018high}
and \emph{extreme inequalities} for symmetric polynomials~\cite{rosset1989normalized} \footnote{The prior work~\cite{schmidt1995chernoff} actually recognized usefulness of symmetry, but was not able to exploit it in the case of general $[-1,1]$-valued random variables.}
\item other applications, in particular sharp bounds for moments of binomial distribution
\end{itemize}

We now move to present our results, adopting the following notation: for two expressions $A,B$ we write $A\lesssim B$ when $A\leqslant K\cdot B$ for some absolute constant $K$, and $A\simeq B$ when the inequality holds in both direction. By $\|Z\|_d = (\mathbb{E}|Z|^d)^{1/d}$ we denote the $d$-th norm of a random variable $Z$. By $\mathbb{E}Z$ and $\mathbb{V}[Z]$ we denote, respectively, the mean and variance of $Z$.

\subsubsection{Sharp Bounds for $k$-wise Independent Sums}

The theorem below gives the complete answer to the posed problem.
\begin{theorem}[Moments of $k$-wise Independent Sums]\label{cor:explicit_bounds}
Consider random variables $(X_i)_{i=1}^{n}$ satisfying the following conditions 
\begin{enumerate}[(a)]
\item $(X_i)_i$ are $k$-wise independent ($k\geqslant 2$) and $|X_i-\mathbb{E}X_i|\leqslant 1$
\item $\sum_{i=1}^{n} \mathbb{V}[X_i]\leqslant n\sigma^2$ (the sum variance bounded)
\end{enumerate}
Then for $S=\sum_{i=1}^{n} X_i$ and any positive even integer $d\leqslant k$ we have:
\begin{align}\label{eq:branches}
\max_{(X_i)} \|S-\mathbb{E}S\|_d \simeq
M(n,\sigma^2,d)=
\begin{cases}
\sqrt{d n \sigma^2} & \log(d/n\sigma^2) < \max(d/n,2) \\
\frac{d}{\log(d/n\sigma^2)} &  \max(d/n,2)\leqslant \log(d/n\sigma^2) \leqslant d \\
(n\sigma^2)^{1/d} &  d<\log(d/n\sigma^2)
\end{cases},
\end{align}
where the maximum is over all r.vs. $(X_i)_i$ satisfying the conditions (a) and (b).

Moreover, the maximal value is realized (up to a constant) when
\begin{align}
X_i\sim B-B',\quad B,B'\sim^{iid}\mathrm{Bern}(p), p=\frac{1}{2}(1-\sqrt{1-2\sigma^2}).
\end{align}
\end{theorem}
\begin{remark}[Value of $k$]
In applications value of $k$ should be possibly big, so that we can use as high moments $d$ as possible. For example, some cryptographic applications use $k = 80$~\cite{dodis2014key}.
\end{remark}
\begin{remark}[Formula Regimes]
The formulas may look a little exotic, especially to a reader familiar with previous works. The reason is that the novel bounds above are sharp and capture some non-standard behaviors.
The branch with $\sqrt{dn\sigma^2}$ should be familiar, as this is the $d$-th moment of gaussian distribution with variance $n\sigma^2$; in this regime it would produce the tail $\Pr[|S-\mathbb{E}S|>t]\leqslant \mathrm{e}^{-\Omega(t^2/n\sigma^2)}$.
The branch with $d/\log(d/n\sigma^2)$ gives the behavior slightly faster than this of the exponential distribution; it would produce the tail $\Pr[|S-\mathbb{E}S|>t]\leqslant \mathrm{e}^{-\omega(t)}$.
Finally, the branch with $(n\sigma^2)^{1/d}$ resembles the behavior of a distribution bounded by $\sigma^2$.
\end{remark}
\begin{remark}[Odd values of $d$]
It can be shown that the same upper bounds apply when $d>2$ is odd, due to interpolation inequalities~\cite{bergh2012interpolation}.
\end{remark}

\begin{remark}[Explicit Tail Bounds]
For $M$ as in \Cref{eq:branches} we obtain the following tail bound (which depends on the parameters regime), for any $t>0$
\begin{align}
\Pr[|S-\mathbb{E}S|>t]\leqslant  O(M(n,\sigma^2,d) / t)^d.
\end{align}
For $t=cM(n,\sigma^2,d)$ with an appropriate constant $c$, we obtain the tail of $2^{-\Omega(d)}$. 
\end{remark}

\subsubsection{Techniques}

We show that for even $d\leqslant k$ we can assume in addition c) full independence and d) full symmetry of the summands,  leveraging \emph{symmetrization}~\cite{vershynin2018high}.
Then we proceed in two steps. 
\paragraph{Characterization of Extreme Distribution}
First, we characterize "worst-case" distributions $X_i$ that maximize $\|\sum_i X_i\|_d$.
This result is the core of our approach and of broader interest, we thus present it as the standalone lemma.

\begin{lemma}[IID Majorization of Symmetric Sums]\label{thm:main1}
Let $(Z_i)_{i=1}^{n}$ be independent symmetric random variables with values in $[-1,1]$ with average variance $\sigma^2 = \frac{1}{n}\sum_{i=1}^{n}\mathbb{V}[Z_i]$.
Then for any positive even integer $d$ we have that
\begin{align}
\|\sum_i Z_i \|_d \leqslant \|\sum_i Z'_i \|_d
\end{align}
where $Z'_i$ are independent and identically distributed as
\begin{align}\label{eq:triple_distribution}
Z'_i \sim\begin{cases}
+1 & \text{w.p. } \sigma^2/2 \\
0 & \text{w.p. } 1-\sigma^2 \\
-1 & \text{w.p. } \sigma^2/2
\end{cases}.
\end{align}
\end{lemma}
\begin{remark}[Interpretation]
Observe that $\mathbb{V}[Z'_i] = 2\cdot\sigma^2/2=\sigma^2$, thus the theorem essentially says that moments of the sum $\sum_i Z_i$ are maximized for $Z_i$ that are iid with the distribution \eqref{eq:triple_distribution}. This three-point distribution is \emph{extreme}, in the sense that it pushes as much mass as possible towards the edge of the interval constraint. 
This behavior may look intuitive, but we should beware of such intuitions as even for simple symmetric problems whether the maximizer's behavior is "push to boundary" or "pull to the middle" may not be that intuitive, depending on Schur convexity properties of the optimized expression~\cite{eaton1982review}; to be specific the problems of maximization of $\sum_{i\not=j}p_i p_j$ and $\sum_i p_i^2$ have quite different behavior. Our proof requires some non-trivial facts about multivariate symmetric polynomials.
\end{remark}

\begin{remark}[Proof Techniques]
The symmetry assumption is crucial, and makes it possible to greatly simplify the multinomial expansion of the moment formula. We manage to regroup expressions and see them as positive combinations of elementary symmetry polynomials; then specialized inequalities from the theory of symmetric functions~\cite{10.2307/43667273} come then to the rescue, allowing for proving that our extreme distribution is indeed the maximizer.
\end{remark}

\paragraph{Closed-Form Bounds for Extreme Distributions}

In addition to characterizing the worst-case behavior, we give the closed-form formula for the bound in \Cref{thm:main1}. As we will see later, this is also a fact of broader interest; for example, we use it to derive bounds for binomial moments which are sharp in all parameter regimes.

\begin{corollary}[Best bounds for IID]\label{thm:best_iid}
For independent ($d$-wise independent) symmetric $Z_i$ with values $\{-1,0,1\}$ and variance $\sigma^2$ and positive even integer $d$ we have 
\begin{align}
    \left\|\sum_{i=1}^{n}Z_i\right\|_d \simeq
\begin{cases}
\sqrt{d n \sigma^2} & \log(d/n\sigma^2) < \max(d/n,2) \\
\frac{d}{\log(d/n\sigma^2)} &  \max(d/n,2)\leqslant \log(d/n\sigma^2) \leqslant d \\
(n\sigma^2)^{1/d} &  d<\log(d/n\sigma^2)
\end{cases}.
\end{align}
\end{corollary}

\begin{remark}[Proof techniques]
The proof requires some effort to compute moments for \Cref{eq:triple_distribution}. Loosely speaking, we leverage the specific form of the distribution to obtain regular combinatorial patterns in multinomial expansions.
We then obtain an explicit formula, being a weighted sum of binomial-like expressions which involve $n$, $d$ and $\sigma$. Establishing the order of growth, with the help of some calculus, completes the proof.
\end{remark}

\subsection{Related Work}

There are many bounds which cover different models of dependencies among random variables, for example Janson's correlation inequality~\cite{janson1998new} which has become very popular in analyses of random graphs~\cite{dubhashi2009concentration}, 
or the theory of negative dependence~\cite{block1982some} best known from applications to various "balls and bins" problems~\cite{dubhashi1996balls}.
However the focus of this paper is on \emph{$k$-wise independence}, in which the state-of-art bounds are due to Schmidt at al.~\cite{schmidt1995chernoff} and Bellare and Rompel~\cite{bellare1994randomness}, derived in the essentially same setup as ours (moment bounds under the variance constraint). These bounds, although useful for many applications, hold only in certain regimes and are not sharp in general; when discussing our applications we will show that in most cases they are inferior to \Cref{cor:explicit_bounds}.

It in our work we consider the most natural variance constraint, following prior works~\cite{schmidt1995chernoff,bellare1994randomness}. However, one might consider more exotic structural assumptions; recently there has been an attempt, limited only to $k=2$, to characterize worst bounds by exploring the whole sequence (rather than the sum variance) of Bernoulli parameters of $X_i$~\cite{ramachandra2020tight}.

Regarding the established concentration bounds, we will see that known inequalities actually imply stronger bounds that those developed by Schmidt et al. and Bellare, Rompel~\cite{schmidt1995chernoff,bellare1994randomness}. However, even with the use of state-of-art moment inequalities~\cite{boucheron2013concentration} we were not able to recover the sharp bounds from our main result and the characterization from \Cref{thm:main1}. The key
challenge is to precisely characterize the worst-case behavior, while allowing \emph{differently distributed} random variables.

An interesting way of attacking the problem, possibly working for much more exotic constraints, may be to formally follow the presented idea of majorization and establish formally some Schur-convexity properties; this approach has been successfully applied in the past to many other problems (see~\cite{eaton1970note} and follow-up works).

\subsection{Applications}


\subsubsection{Limited Independence: Clarifying the State-of-Art}

We will demonstrate how our bounds improve on those of Schmidt at al.~\cite{schmidt1995chernoff} and Bellare and Rompel~\cite{bellare1994randomness}, clarifying this way the state-of-the-art. In what follows we assume, as in our theorem, that $\|X_i-\mathbb{E}X_i|\leqslant 1$, $X_i$ are $k$-wise independent, and $d\leqslant k$ for positive even $d$.
The best bound due to Schmidt at al. reads as (cf Eq. 10 in~\cite{schmidt1995chernoff})
\begin{align*}
\|S-\mathbb{E}S\|_d^d \leqslant \sqrt{2}\cdot \mathrm{cosh}(\sqrt{d^3/36C})\cdot (dC/\mathrm{e})^{d/2},\quad C\geqslant n\sigma^2,\ \sigma^2 = \mathbb{V}[S]/n.
\end{align*}
The authors did not fully optimize the choice of $C$, offering a bunch of weaker corollaries instead. In order to clarify the state-of-the-art, we do this effort (see \Cref{sec:schmidt_optimized}) obtaining
\begin{align}
\|S-\mathbb{E}S\|_d \lesssim \max(\sqrt{dn\sigma^2},d). 
\end{align}
When $dn\sigma^2\geqslant 1$ the formula matches ours, but otherwise it is much worse: by a factor of $\log(d/n\sigma^2)$ in the regime $\max(2,d/n)\leqslant \log(d/n\sigma^2)\leqslant d$, and by a factor
of $d$ in the regime $d< \log(d/n\sigma^2)$ (which necessarily means $n\sigma^2<1$). In applications, these factors can be a big constant or more, so derived tail bounds are worse by a big constant in the exponent.

In turn, the bound due to Bellare and Rompel~\cite{bellare1994randomness} states that when $X_i\in[0,1]$ 
\begin{align}\label{eq:bellare_bound}
\|S-\mathbb{E}S\|_d \lesssim \min(\sqrt{d n }, \sqrt{d n \mu + d^2}),\quad \mu = \frac{1}{n}\mathbb{E}S.
\end{align}
We claim this is worse than our optimized version of Schmidt at al., in all regimes. Namely,
\begin{align*}
\max(\sqrt{dn\sigma^2},d) \lesssim \min(\sqrt{d n }, \sqrt{d n \mu + d^2}). 
\end{align*}
Indeed, when $d>n$ the left-hand side is at most $n$, while the right-hand side is at least $n$. When $d\leqslant n$, due to $\mu\leqslant 1$ (a consequence of $X_i\leqslant 1$) we see that
$\min(\sqrt{d n }, \sqrt{d n \mu + d^2})\simeq \sqrt{dn\mu+d^2}$. But we have $\mathbb{V}[X_i]\leqslant \mathbb{E}X_i $, as the consequence of $0\leqslant X_i\leqslant 1$, and thus $\sigma^2\leqslant \mu$;
this shows $\min(\sqrt{d n }, \sqrt{d n \mu + d^2}) \gtrsim \sqrt{dn\sigma^2}$ and the claim follows.

This discussion should be of broader interest to the TCS community, as it seems that no rigorous comparison between~\cite{bellare1994randomness} and~\cite{schmidt1995chernoff} has been done before (the surveys such as~\cite{Abbas2011} and application works credit both exchangably). In \Cref{tab:summary} we give a readable summary.

\begin{table}
\captionsetup{font=footnotesize}
\centering{
\resizebox{0.98\textwidth}{!}{
\def\arraystretch{1.3}
\begin{tabular}{|l|l|l|}
\hline
\rowcolor{Gainsboro!60}
Bound on $\|S-\mathbb{E}S\|_d$ & Author & Assumptions \\
\hline
$\max(\sqrt{dn\sigma^2},d/\log(d/n\sigma^2),(n\sigma^2)^{1/d})$  & \textbf{this paper} & $n\sigma^2 = \mathbb{V}[S], \ |X_i-\mathbb{E}X_i|\leqslant 1$\\
\hline
$ \max(\sqrt{dn\sigma^2},d)$ & Schmidt at al., \textbf{optimized} & $n\sigma^2 = \mathbb{V}[S],  |X_i-\mathbb{E}X_i|\leqslant 1$ \\
\hline
$\min(\sqrt{d n }, \sqrt{d n \mu + d^2}) $ & Bellare and Rompel & $n\mu = \mathbb{E}[S],\ 0\leqslant X_i\leqslant 1$ \\ 
\hline
\end{tabular}
}
}
\bigskip
\caption{Bounds for moments of sum $S=\sum_{i=1}^{n}X_i$ of $d$-wise independent random variables, where $d\geqslant 2$. As discussed, our bounds are strictly better than those of Schmidt at al., which in turn are strictly better than those of Bellare and Rompel.}
\label{tab:summary}
\end{table}

\subsubsection{Obtaining Previous Results form Classical Inequaliteis}

Our literature search shows, perhaps surprisingly, that the optimized bounds of Schmidt at. al are actually a \emph{simple consequence of
classical inequalities}; we note that the prior works~\cite{bellare1994randomness,schmidt1995chernoff} do not discuss the related literature on concentration bounds. The intent of this discussion is to bring those inequalities to the awareness of the wider TCS audience, particularly given the huge interest and  the citation credit given to the bounds in~\cite{bellare1994randomness,schmidt1995chernoff}.

Assume that $X_i$ are $k$-wide independent; recall that the event moment of order $d\leqslant k$ can be calculated as if the summands were independent. More precisely, we have $\sum_i \|X_i -\mathbb{E}X_i\|_d = \sum_i \|X'_i -\mathbb{E}X'_i\|_d$ where $X'_i$ are distributed as $X_i$ and independent. The tail bounds due to a century old (!) \emph{Bernstein's inequality}~\cite{bernsteinprobability} imply that
the tail of $S'=\sum_i X_i'$ satisfies $\Pr[|S'-\mathbb{E}S'|>t] \leqslant \exp(-\Theta(\min(t^2/n\sigma^2,t)))$ for any positive $t$,
if $X_i$ are bounded, and $n\sigma^2 = \sum_i\mathbb{V}[ X'_i]=\sum_i \mathbb{V}[X'_i]$. By the standard tail integration formula,
we find that the moments of the IID sum are $\|S'-\mathbb{E}S'\|_d \lesssim \max(\sqrt{nd\sigma^2},d)$. As remarked, this matches $\|S-\mathbb{E}S\|_d$ when $d<k$, so we recover the optimized (!) bounds of Schmidt at al., and implies the bounds of Bellare and Rompel.
Another argument can be given by the use of \emph{Rosenthal's inequality}, a version of which~\cite{figiel1997extremal} implies $
\|S-\mathbb{E}S\|_d \lesssim d\cdot (\sum_{i=1}^{n} \mathbb{E}|X_i-\mathbb{E}X_i]^{d} )^{1/d} + d^{1/2}\cdot (\sum_{i=1}^{n} \mathbb{E}|X_i-\mathbb{E}X_i]^{2} )^{1/2}$. This can be further bounded by $\max(\sqrt{dn\sigma^2},d)$.

\subsubsection{Sharp Explicit Bounds on Binomial Moments}

Somewhat surprisingly, to the best of author's knowledge, there are no good closed-form estimates on moments of the binomial distribution, despite the clear demand from applications (such as the analysis of random projections~\cite{allen2014sparse,jagadeesan2017simple}). 
The sharp (up to an $o(1)$ relative error term) tail bounds due to Littlewood~\cite{littlewood1969probability,mckay1989littlewood} in theory imply sharp moment estimates, but calculations lead to very difficult integrals with Kullback-Leibler divergence in the exponent.
We obtain closed-form bounds for even binomial moments as a byproduct of our analysis, which are sharp in all paramater regimes.
More precisely, we have
\begin{corollary}
Let $S\sim\mathrm{Binom}(n,p)$ where $p\leqslant 1/2$ and $d$ be a positive even integer. Then
\begin{align}
    \left\|S-\mathbb{E}S\right\|_d \simeq
\begin{cases}
\sqrt{d n p} & \log(d/np) < \max(d/n,2) \\
\frac{d}{\log(d/np)} &  \max(d/n,2)\leqslant \log(d/np) \leqslant d \\
(np)^{1/d} &  d<\log(d/np)
\end{cases}.
\end{align}
\end{corollary}
This follows from the fact that the extreme variables $Z'_i$ in our main result can be expressed as symmetrized Bernoulli distributions, namely $Z'_i  \sim B-B'$ where $B,B'\sim^{iid}\mathrm{Bern}(p)$ with $p=\frac{1}{2}(1-\sqrt{1-2\sigma^2})$. Let $S\sim \mathrm{Binom}(n,p)$. By symmetrization $\|S-\mathbb{E}S\|_d\simeq \|S-S'\|_d$ and thus  $\|S-\mathbb{E}S\|_d \simeq \|\sum_i (B_i-B_i')\|_d \simeq \|\sum_i  Z'_i\|_d$, thus by our result $\|S-\mathbb{E}S\|_d$ obeys the bound as above with $p$ replaced by $\sigma^2$. It remains to observe that $\sigma^2 = 2p(1-p)$ so $p\leqslant \sigma^2\leqslant 2p$, and that the bounds above do not change by a more than a constant when $p$ is replaced by $p'\in [p,2p]$.

\subsubsection{Exact Binomial Moments}

Binomial moments can be evaluated by means of combinatorics, which yields somewhat complicated recursions~\cite{knoblauch2008closed}. Interestingly, a byproduct of \Cref{thm:best_iid} gives an exact formula for \emph{symmetrized} binomials which has very simple form.

\subsubsection{Estimating binomial-like moments}

The line of research focused on estimating Renyi entropy of unknown probability distributions~(\cite{acharya2016estimating,obremski2017renyi}) faces the problem of estimating moments of sum of random variables in "small variance" regime, that is when $n\sigma^2 \ll 1$. For example, the collision estimator requires bounds on the $4$-th sum moment.
This has been previously done by exploiting somewhat tedious combinatorial identities, but follows easier from \Cref{thm:best_iid} and \Cref{thm:main1}.

\section{Preliminaries}

\subsection{Multinomial Expansion}
The \emph{multinomial coefficient} is defined as 
\begin{align}
\binom{d}{\mathbf{j}} = d! / \prod_{j\in\mathbf{j}} j!
\end{align}
when all components of $\mathbf{j}$ are non-negative and $\sum_{j\in\mathbf{j}}j=d$. We also extend this to $\binom{d}{\mathbf{j}}=0$ when $\min\{j:j\in \mathbf{j}\} < 0$ or $\sum_{j\in\mathbf{j}}\not = d$; this allows for concise notation. The multinomial formula takes the form 
$(\sum_i x_i)^d = \sum_{\mathbf{i}} \prod_{i\in\mathbf{i}} x^{i}$.
\begin{remark}
Factorials, and therefore binomial coefficients can be formally extended to negative numbers by means of Gamma function. Then indeed multinomial coefficients are zero when negative integers appear as downward arguments~\cite{kronenburg2011binomial}.
\end{remark}
We will occasionally use the Stirling's formula in estimation of multinomial coefficients~\cite{nemes2010coefficients,robbins1955remark}
\begin{align}
(d/\ell)! \simeq \sqrt{d/\ell}\cdot(d/\mathrm{e}\ell)^{d/\ell}.
\end{align}

\subsection{Symmetrization}
We will need the following facts about symmetrization (cf \cite{vershynin2018high}).
\begin{proposition}[Convex Symmetrization]\label{lemma:symmetrization}
For zero-mean iid $X,X'$ and convex $f$
\begin{align}
\mathbb{E}f(X) \leqslant \mathbb{E}f(X-X').
\end{align}
\end{proposition}
\begin{proof}[Proof of \Cref{lemma:symmetrization}]
By independence $\mathbb{E}f(X-X') = \mathbb{E}_{X}\mathbb{E}_{X'}[f(X-X')|X]$ and by Jensen's inequality $\mathbb{E}_{X'}[f(X-X')|X] \geqslant f(X-\mathbb{E}X')$. By the zero-mean assumption $ f(X-\mathbb{E}X') = f(X)$. 
The inequality follows by chaining these three bounds.
\end{proof}
\begin{proposition}[Moments are robust under symmetrization]\label{prop:symmetrize_moments}
For any iid random variables $\|X\|_d\leqslant \|X-X'\|_d \leqslant 2\|X\|_d$
\end{proposition}
\begin{proof}[Proof of \Cref{prop:symmetrize_moments}]
Since $\|X\|_d = (\mathbb{E}|X|^d)^{1/d}$, the left-hand side follows by applying \Cref{lemma:symmetrization} to $f(u)=|u|^{d}$. The right-hand side is due to the triangle inequality (Minkovski's inequality for $L_p$ spaces).
\end{proof} 

\subsection{Symmetric Functions}

The $\ell$-th elementary symmetric polynomial in variables $u=(u_i)_i$ is defined as
\begin{align}
\Pi_{\ell}(u)=\sum_{i_1<\ldots<i_{\ell}} u_{i_1}u_{i_2}\cdot\ldots\cdot u_{i_{\ell}}.
\end{align}
The fundamental theorem on symmetric polynomials states that they generate all other symmetric polynomials (in a sense of the algebraic ring)~\cite{daoub2012fundamental}. We will need some facts about their extreme properties, which we recall below.

\begin{proposition}[Newton Inequalities~\cite{newton1761arithmetica}]
For $u=(u_i)_{i=1}^{n}$ let $S_{\ell}(u) \triangleq \Pi_{\ell}(u) / \binom{n}{\ell}$ be the $\ell$-th elementary symmetric mean. Then
$S_{\ell-1}(u)S_{\ell+1}(u)\leqslant S_{\ell}(u)^2$.
\end{proposition}
This implies the useful inequality due to Maclaurin
\begin{proposition}[Maclaurin's Inequality~\cite{maclaurin1729second,niculescu2006convex}.]\label{prop:maclauren}
For $u=(u_i)_{i=1}^{n}$ we have the inequality $S_{\ell}(u)^{1/\ell}\geqslant S_{\ell'}(u)^{1/\ell'}$ when $1\leqslant \ell<\ell'\leqslant n$ (with the equality when $u_i$ are equal).
\end{proposition}

\section{Proofs}



\subsection{Proof of \Cref{thm:main1}}

We use the fact that $d$ is an even integer and the multinomial formula to expand
\begin{align}
    \mathbb{E}\left|\sum_{i=1}^{n}Z_i\right|^d = 
        \mathbb{E}\left(\sum_{i=1}^{n}Z_i\right)^d = 
        \sum_{\mathbf{j}} \binom{d}{\mathbf{j}} \mathbb{E}[ Z_{1}^{\mathbf{j}_1}\cdot\ldots\cdot Z_{n}^{\mathbf{j}_n} ]
\end{align}
The summation is over integer tuples $\mathbf{j} \in \mathbb{Z}^n$
called also multiindices. Utilizing the independence assumption, we obtain
\begin{align}
    \mathbb{E}\left|\sum_{i=1}^{n}Z_i\right|^d =
        \sum_{\mathbf{j}} \binom{d}{\mathbf{j}} \mathbb{E}[ Z_{1}^{\mathbf{j}_1}]\cdot\ldots\cdot \mathbb{E}\mathbf[Z_{n}^{\mathbf{j}_n} ]
\end{align}
Since $Z_i$ are symmetric, all odd moment vanish. Thus,
we can write
\begin{align}
    \mathbb{E}\left|\sum_{i=1}^{n}Z_i\right|^d =
        \sum_{\mathbf{j}} \binom{d}{2\mathbf{j}} \mathbb{E}[ Z_{1}^{2\mathbf{j}_1}]\cdot\ldots\cdot \mathbb{E}\mathbf[Z_{n}^{2\mathbf{j}_n} ].
\end{align}
Since $Z_i$ are absolutely bounded by $1$ and symmetric, we have $\mathbb{E}|Z_i|^{2j}\leqslant \mathbb{E}|Z_i|^2\leqslant \mathbb{V}[Z_i]$ for $j\geqslant 1$. Denoting $\sigma^2_i = \mathbb{V}[Z_i]$ we can write
\begin{align}
    \mathbb{E}\left|\sum_{i=1}^{n}Z_i\right|^d \leqslant
        \sum_{\mathbf{j}} \binom{d}{2\mathbf{j}}\prod_{i: \mathbf{j}_i \not= 0 } \sigma_i^2. 
\end{align}
\begin{remark}
The equality is met when $Z_i$ are symmetric with values in the set $\{-1,0,1\}$, as this implies $\mathbb{E}|Z_i|^j = \mathbb{E}|Z_i|^2$.
\end{remark}
Let $\|\mathbf{j}\|_0 = \#\{i:\mathbf{j}_i\not=0\}$ be the number of non-zero indices in the multiindex $\mathbf{j}$. Clearly $\ell=\|\mathbf{j}\|_0$ can take values from $1$ to $\frac{d}{2}$ and thus
\begin{align}\label{eq:my1}
    \mathbb{E}\left|\sum_{i=1}^{n}Z_i\right|^d \leqslant \sum_{\ell=1}^{d/2}\underbrace{\sum_{\mathbf{j}:\|\mathbf{j}\|_0 =\ell} \binom{d}{2\mathbf{j}}\prod_{i: \mathbf{j}_i \not= 0 } \sigma_i^2}_{S_{\ell}}. 
\end{align}
Note that $S_{\ell}$ is multilinear of order $\ell$ in $u_i=\sigma_i^2$.
We claim that it equals the elementary symmetric polynomial, up to a constant multiplier (this is not clear a-priori as different weights could break the symmetry).
\begin{claim}
The polynomial $S_{\ell}$ is a (non-negative) multiplicity of the $\ell$-th elementary symmetric polynomial $\Pi_{\ell}$ in variables $\sigma_i^2$.
\end{claim}
\begin{proof}[Proof of Claim]
Indeed, consider $S_{\ell}$ as the weighted sum of monomials $\prod_{i\in I} \sigma_i^2$, where $\|I\|=\ell$. Every such a monomial appears with the coefficient $c_I \triangleq \sum_{\mathbf{j}:\mathbf{j}_i\not=0 \Leftrightarrow i\in I} \binom{d}{2\mathbf{j}}$. Due to the symmetry of the multinomial coefficient $\binom{d}{2\mathbf{j}}$, namely the invariance under permuting $\mathbf{j}$, we claim that $c_I$ is the same for every set $I$. Indeed, if $\rho(I')=I$ for a bijection $\rho$ then 
\begin{align}
    c_{I'} = \sum_{\mathbf{j}:\mathbf{j}_{i'}\not=0 \Leftrightarrow i'\in I'} \binom{d}{2\mathbf{j}} 
= 
    \sum_{\mathbf{j}:\mathbf{j}_{\rho(i)}\not=0 \Leftrightarrow \rho(i)\in I} \binom{d}{2\mathbf{j}} 
= 
     \sum_{\mathbf{j}:\mathbf{j}_{\rho(i)}\not=0 \Leftrightarrow \rho(i)\in I} \binom{d}{2\mathbf{\rho(j)}} 
= c_I
\end{align}
It follows that $S_{\ell}$ is a multiplicity of the $\ell$-th elementary symmetric polynomial (as it contains all monomials of order $\ell$ with equal coefficients). This proves the claim.
\end{proof}
We now establish extreme properties of $S_{\ell}$. Namely
\begin{claim}
The expression $S_{\ell}$ is maximized, subject to the constraint that $\sum \sigma_i^2$ is kept constant, when all $\sigma_i$ are equal. 
\end{claim}
\begin{proof}[Proof of Claim]
This follows by Maclaurin's Inequality in~\Cref{prop:maclauren}.
\end{proof}
Let $\sigma^2 = \frac{1}{n}\sum_i\sigma_i^2$,  from the claim we obtain
\begin{align}
    \mathbb{E}\left|\sum_{i=1}^{n}Z_i\right|^d \leqslant \sum_{\ell=1}^{d/2}\sum_{\mathbf{j}:\|\mathbf{j}\|_0 =\ell} \binom{d}{2\mathbf{j}}\cdot  \sigma^{2\ell}
\end{align}
The right-hand side is like in 
\Cref{eq:my1} with all $\sigma_i$ equal, and by the remark we know that it equals $\mathbb{E}\left|\sum_{i=1}^{n}Z'_i\right|^d$ if $Z'_i$ is symmetric with values $\{-1,0,1\}$ and has variance $\sigma^2$.

\subsection{Proof of \Cref{thm:best_iid}}

For $Z_i$ as in \Cref{thm:main1} the previous section derives the identity
\begin{align}
    \mathbb{E}\left|\sum_{i=1}^{n}Z_i\right|^d = \sum_{\ell=1}^{d/2}\sum_{\mathbf{j}:\|\mathbf{j}\|_0 =\ell} \binom{d}{2\mathbf{j}}\cdot  \sigma^{2\ell}.
\end{align}
We will further simplify this expression. Considering positive components of $\mathbf{j}$ we obtain
\begin{align}
    \sum_{\mathbf{j}:\|\mathbf{j}\|_0 =\ell} \binom{d}{2\mathbf{j}} = \sum_{\ell=1}^{d/2}\sum_{i_1 <\ldots < i_{\ell}} \sum_{\mathbf{j}: j_i\geqslant 1\Leftrightarrow i\in \{i_1,\ldots,i_{\ell}\}} \binom{d}{2\mathbf{j}}\cdot \sigma^{2\ell} .
\end{align}
Since the expression is invariant under permutations of $\mathbf{j}$ we obtain
\begin{align}
    \sum_{\mathbf{j}:\|\mathbf{j}\|_0 =\ell} \binom{d}{2\mathbf{j}}  = \sum_{\ell=1}^{d/2}\sum_{\mathbf{j}=(j_{1},\ldots,j_{\ell})\geqslant 1} \binom{d}{2\mathbf{j}} \cdot \binom{n}{\ell} 
\end{align}
where $\binom{n}{\ell}$ counts the number of choices for $i_{1},\ldots,{i_{\ell}}$. Therefore
\begin{align}
    \mathbb{E}\left|\sum_{i=1}^{n}Z_i\right|^d = \sum_{\ell=1}^{d/2}  \sum_{\mathbf{j}=(j_{1}\ldots j_{\ell})\geqslant 1} \binom{d}{2\mathbf{j}} \binom{n}{\ell} \sigma^{2\ell}.
\end{align}
We now estimate the moment up to constants. Since for any positive $a_i$ we have $(\sum_{i=1}^{d} a_i)^{1/d} \simeq (\max_i a_i)^{1/d}$ up to some absolute constants (in fact, constants are $1$ and $d^{1/d}\leqslant 2$), we obtain
\begin{align}\label{eq:moment_3}
   \left\|\sum_{i=1}^{n}Z_i\right\|_d \simeq \max_{\ell=1,\ldots,d/2}\left[ \underbrace{ \sum_{\mathbf{j}=(j_{1},\ldots,j_{\ell})\geqslant 1} \binom{d}{2\mathbf{j}} }_{F(\ell)} \cdot \binom{n}{\ell} \cdot  \sigma^{2\ell} \right]^{1/d}
\end{align}
In the next step we estimate $F(\ell)^{1/d}$.
\begin{claim}\label{cor:multinomial_bounds}
We have $F(\ell)^{1/d} \simeq \ell$.
\end{claim}
\begin{proof}
For the lower bound we can assume that $\ell$ (and hence $d$) are sufficiently big (otherwise the bound is trivial).
We can also assume that $\ell$ divides $d$ and that $r=\lfloor d/\ell\rfloor$ is even; otherwise we replace $\ell$ with $\ell'$ between $\ell$ and $\ell/2$ which satisfies this, use 
$F(\ell)\geqslant F(\ell')$ and prove for $\ell'$.
Consider the term $2j_1=\ldots 2j_{\ell-1}=d/\ell$, we have
\begin{align*}
    F(\ell) \geqslant \binom{d}{r,\ldots,r} = \frac{d!}{(r!)^{\ell}}.
\end{align*}
Observe that $(d!)^{1/d}\simeq d$, $(r!)^{1/r} \simeq r$ and $(q!)^{1/q}\simeq q$ (Stirling's formula).
Since $r\cdot \frac{\ell}{d} \leqslant 1$ we get $(r!)^{\frac{\ell}{d}} \simeq r^{r\cdot\frac{\ell}{d}}=r$ (the relation $\simeq$ can be raised to a bounded power). This gives
\begin{align*}
F(\ell)^{1/d} \gtrsim \frac{d}{r} =\ell.
\end{align*}
As for the upper bound, we simply note that 
\begin{align*}
F(\ell)^{1/d} < \left(\sum_{\mathbf{j}=(j_1,\ldots,j_{\ell})}\binom{d}{\mathbf{j}}\right)^{1/d} \leqslant (\ell^d)^{1/d} = \ell.
\end{align*}
These two bounds completes the proof.
\end{proof}
By \Cref{cor:multinomial_bounds} and \Cref{eq:moment_3} we obtain the following, much simpler bound
\begin{align}\label{eq:sup4}
    \left\|\sum_{i=1}^{n}Z_i\right\|_d \simeq \max_{\ell=1,\ldots,d/2}\left[\ell^{d}\cdot\binom{n}{\ell}\cdot\sigma^{2\ell}\right]^{1/d}.
\end{align}
With some more effort we simplify even further. Namely, we can assume $\ell\leqslant n$ as for $\ell>n$ we have $\binom{n}{\ell} = 0$.
By the elementary inequality $(n/\ell)^{\ell} \leqslant \binom{n}{\ell} \leqslant (n\mathrm{e}/\ell)^{\ell}$ we have $\binom{n}{\ell}^{1/d} \simeq (n/\ell)^{\ell/d}$, for $\ell=1\ldots d/2$. Thus
\begin{align}\label{eq:sup5}
\begin{split}
    \left\|\sum_{i=1}^{n}Z_i\right\|_d &\simeq \max_{\ell=1,\ldots,\min(d/2,n)}\left[\ell^{d}\cdot\binom{n}{\ell}\cdot\sigma^{2\ell}\right]^{1/d} \\
&\simeq \max_{\ell=1,\ldots,\min(d/2,n)}\left[\ell\cdot (n/\ell)^{\ell/d} \cdot\sigma^{2\ell/d}\right] 
\end{split}.
\end{align}
Losing not more than a constant factor, we can extend the maximum to the continuous interval (the expression under maximum differs by at most a constant factor between two values of $\ell$ that differ by one or less).
Let $q=d/\ell$, we have the equivalent constraint $\max(2,dn/n)\leqslant q\leqslant d$ and the maximum of $d/q\cdot (n\sigma^2/d)^{1/q}\cdot q^{1/q}$. Since $q^{1/q}\simeq 1$ when $q\geqslant 1$
\begin{align}\label{eq:sup6}
    \left\|\sum_{i=1}^{n}Z_i\right\|_d &\simeq \max_{q:\max(2,d/n)\leqslant q\leqslant d}\left[d/q\cdot (n\sigma^2/d)^{1/q}\right] .
\end{align}
It now suffices to analyze the auxiliary function $g(q)\triangleq 1/q\cdot a^{1/q}$ for $q>0$. The derivative test shows that it is decreasing when $a>1$ and has the global maximum at $q=\log(1/a)$ with the value $1/\mathrm{e}\log(1/a)$
when $a<1$. This behavior is illustrated on~\Cref{fig:aux_g}.
\begin{figure}[ht]
\centering
\begin{tikzpicture}[scale=1]
  \begin{axis}[
      samples=200,
      axis lines=middle,
    legend style= {at = {(1,1)}, anchor=north east},
    ytick = {0.53},
    yticklabels = {$\frac{1}{\mathrm{e}\log(1/a)}$},
    ymax = 0.7,
    restrict x to domain=0:5,
    xtick = {0,1.44},
    xticklabels = {0,$\frac{1}{\log(1/B)}$}
    ]
    \addplot [domain=0:5] {x*(0.5)^(x)};
    \addlegendentry{$g(q) = 1/q \cdot a^{1/q}$};
  \end{axis}
\end{tikzpicture}
\caption{Auxiliary function $g$ which determines the moment behavior (for $a<1$).}
\label{fig:aux_g}
\end{figure}
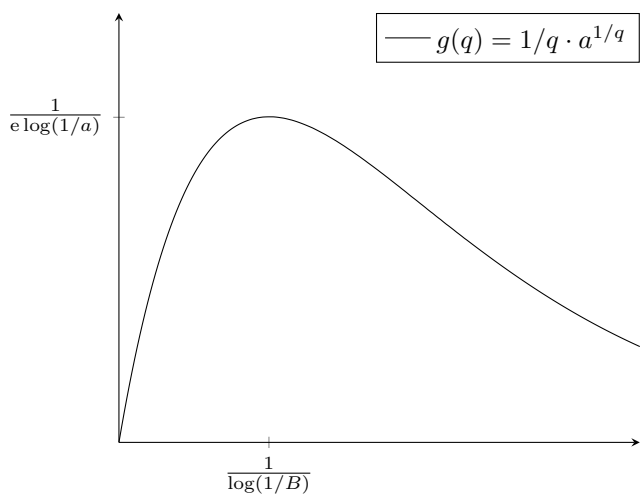

Applying this fact to \Cref{eq:sup6}, with $a=n\sigma^2/d$, and comparing $\log(1/a)$ with the interval boundaries finishes the proof.

\subsection{Proof of \Cref{cor:explicit_bounds}}

We recall the folklore fact that the moment of order $d\leqslant k$ of the sum of $k$-wise independent r.vs. can be computed as if they were independent. That is, let $X'_i$ be distributed as $X_i$ but independent. For even $d$ we have
\begin{align}
\mathbb{E}|\sum_i X_i|^d = \mathbb{E}|\sum_i X;_i|^d
\end{align} 
which follows by applying the multinomial expansion on both sides and observing that the obtained formulas depend only on products 
of at most $d$ of random variables $X_i$ (respectively $X'_i$). Without loss of generality, we can also assume that $X_i$ are centered.
Now the result follows if $X_i$ are symmetric, by \Cref{thm:main1} and \Cref{thm:best_iid}.
If they are not symmetric, we can use the general reduction as in \Cref{prop:symmetrize_moments};
namely, we apply the proof to $X'_i-X_i''$ where $X'_i,X''_i\sim^{iid} X_i$. The moments differ by at most a factor of two.
Particularly, the variance changes by a factor of 2, which has no impact on the asymptotic bounds in \Cref{eq:branches}. More precisely, we use the fact that $M(n,\sigma^2,d) \simeq M(n,{\sigma'}^2,d)$ where $\sigma^2/2\leqslant \sigma^2 \leqslant 2{\sigma'} ^2$.

\subsection{Optimized moment bound of Schmidt at al.}\label{sec:schmidt_optimized}


Up to a constant factor, their bound is equivalent to
\begin{align}\label{eq:schmidt_rewritten}
\|S-\mathbb{E}S\|_d \lesssim \mathrm{cosh}(\sqrt{d/36C}) \sqrt{dC},\quad \text{for any }C\geqslant \mathbb{V}[S].
\end{align}
Let $t = \sqrt{d/36C}$, the the upper bound is equivalent to $c\cdot d\cdot \mathrm{cosh}(t)/ t$ where $c$ is an absolute constant. 
We use this to function understand the behavior of \Cref{eq:schmidt_rewritten} on $C$, which is as illustrated on \Cref{fig:schmidt}.
\begin{figure}[t!]
\centering
\begin{tikzpicture}[scale=0.75]
  \begin{axis}[
      samples=200,
      enlarge y limits=false,
     enlarge x limits = false,
      axis lines=middle,
      xlabel = $C$,
    legend style= {at = {(1,1)}, anchor=north east},
    ytick = {-2,1.51},
    yticklabels = {$1.5d$},
    restrict y to domain=-2:10,
    xtick = {0,1.2},
    xticklabels = {0,$C^{*}\approx 0.696 \cdot d/36$},
   clip = false
    ]
    \addplot [domain=0:3] {5+cosh(x)*1/x};
    \addlegendentry{$\mathrm{cosh}(\sqrt{d/36C}) \sqrt{dC}$};
  \end{axis}
\end{tikzpicture}
\caption{The moment bound of Schmidt at al., dependency on $C$.}
\label{fig:schmidt}
\end{figure}
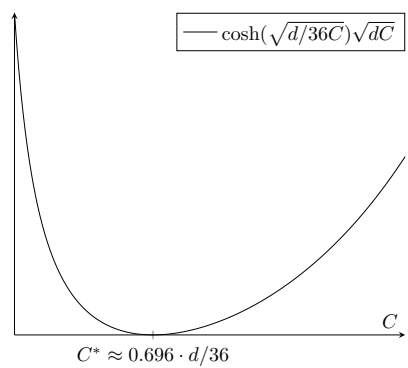
Since $C$ is constrained by $C\geqslant \mathbb{V}[S]$, the best bound is obtained
for 
\begin{align}
C = \max(C^{*},\mathbb{V}[S]),
\end{align}
which gives the claimed bound of
\begin{align}
\|S-\mathbb{E}S\|_d \lesssim \max(\sqrt{dn\sigma^2},d).
\end{align}

\section{Conclusion}

We have developed sharp estimates on the moments of (non necessarily identically distributed) sums of random variables, assuming the variance is constrained. This essentially closes the problem of establishing good concentration bounds, discussed in prior works.
Our approach demonstrates the power of \emph{symmetrization} technique, and is of independent interest. We also showed applications, not limited to $k$-wise independence.


\bibliography{citations}

\appendix

\end{document}